\documentclass{article}
\usepackage{amsmath}
\usepackage{mathtools}
\usepackage{amsfonts}
\usepackage{mathrsfs}
\usepackage{amsthm}
\usepackage{fullpage}
\title{Hypocoercivity in Wasserstein-1 for the kinetic Fokker-Planck equation via Malliavin Calculus.}
\author{Josephine Evans}

\newtheorem{thm}{Theorem}

\newtheorem{lemma}{Lemma}
\newtheorem{prop}{Proposition}
\newtheorem*{remark}{Remark}
\newtheorem{assumption}{Assumption}

\begin{document}
\maketitle
\begin{abstract}
We study the kinetic Fokker-Planck equation on the whole space with a confining potential. We show quantitative rates of exponential convergence to equilibrium in a well chosen Wasserstein-1 distance. We use the Wasserstein-1 version of Harris's theorem introduced by Hairer and Mattingly. We make use of similarities between hypocoercivity and hypoellipticity in order to use Malliavin calculus to see hypocoercivity for this equation on the level of the SDE.
\end{abstract}
\section{Inroduction}

Hypocoercivity was introduced by Villani in \cite{V09}. An equation is hypocoercive if we can show quantitative exponential rates of convergence of a solution of an equation, $f(t)$, towards equilibrium, $\mu$, of the form
\[ d(f(t),\mu) \leq C e^{-\lambda t} d(f(0), \mu),  \] where $C, \lambda$ are explicitly computable strictly positive constants. Hypocoercivity is almost always studied in the context of spatially inhomogeneous kinetic equations.

 In this paper we look at one of the first equations which was studied in the context of hypocoercivity, the kinetic Fokker-Planck or Langevin equation
\[ \partial_t f + v \cdot \nabla_x f - \nabla_x U \cdot \nabla_v f = \Delta_v f + \nabla_v \cdot (vf). \] Here $\mu = M \exp(-|v|^2/2 - U(x))$ for some normalising constant $M$. Hypocoercivity for the kinetic Fokker-Planck equation has been shown by many authors. It was shown in $L^2(\mu^{-1})$ in \cite{HN04}. This paper then inspired the m\'{e}moire of Villani \cite{V09} where he proves a general theorem in the first section which he then applies to the kinetic Fokker-Planck. The $L^2$ and $H^1$ results are also given as special cases of the theorems proven in \cite{DMS15} and \cite{NM06} respectively.

The kinetic Fokker-Planck equation is an equation in the sum of squares form given in \cite{V09} with $B= v \cdot \nabla_x - \nabla_x U \cdot \nabla_v$ and $A = -\nabla_v $. Then
\[ \partial_t f + Bf +A^*Af = 0.  \]
This equation is also hypoelliptic. The hypocoercivity and hypoellipticity of some degenerate diffusions can be proved using similar techniques and the name hypocoercivity was inspired by this similarity. The main examples of this is the paper \cite{HN04} where they prove hypocoercivity and hypoellipticity simultaneously using pseudo differential techniques and the new proof of hypoellipticity for the kinetic Fokker-Planck equation given in \cite{H17}. The link is expressed clearly in \cite{H17}. These proofs of both hypocoercivity and hypoellipticity for kinetic Fokker-Planck equation use crucially the fact that
\[ [B,A] = - \nabla_x.\]  More generally both hypocoercivity and hypoellipticity rely on the diffusion being spread to the other direction seen by taking successive iterated commutators between the vector fields \cite{H67}.

 Some degenerate diffusions equations are also the Kolmogorov backwards equations for the law of the SDE
\[ \mathrm{d}Z_t = \sum_i\tilde{A}_i \mathrm{d}W^i_t + \tilde{B} \mathrm{d}t. \] Where the tilde vector fields are closely related to the the ones appearing in the PDE. In \cite{V09} (Part 1, Prop 5) Villani shows that all SDEs which converge to an equilibrium state have backwards equations which can be written in the form
\[ \partial_t f + \sum_i A_i^* A_i f + B f = 0.  \] This is the form for which it is possible to state his hypocoercivity theorem. Here the vector fields are different to those in the It\={o} SDE form of the equation. Hypoellipticity has been understood on the level of SDEs via Malliavin calculus see for example \cite{M78, N86}. The machinery of Malliavin calculus allow one to see how the effect of the Brownian motions is transferred along different directions given by the iterated commutators of the driving vector fields.

Kinetic Fokker-Planck equations were shown to converge to equilibrium in \cite{HMS02} using techniques from \cite{MT93}. These works use probabilistic techniques, relying on Harris's Theorem which gives exponential convergence to equilibrium based on a Lyapunov condition and a minorization condition. The minorization condition is typically of the form that for all $R$ there exists some probability measure $\nu$ and constant $\alpha$ such that for all $z$ in $B(0,R)$ we have
\[ f_t^z \geq \alpha \nu. \] Here $f_t^z$ is the solution to the PDE at time $t$, with initial condition $\delta_z$.

These proofs do not give explicit constants and this lack of quantifiability arises when showing the minorisation condition. They first show that $f_t^z$ has a density using hypoellipticity theory. Then they show via control theory that for some compact $C$ then there is some $y \in C$ such that for any $\delta$ we have $t_1(\delta)$ with
\[\mathcal{P}_{t_1}(x, B_\delta(y)) >0 \hspace{10pt} \forall x \in C. \] They then use these to prove a minorisation condition. Its not clear how to make this argument quantitative as it would require us to be able to estimate $p_t(x,y)$ from below at a specific point and uses compactness arguments. As the proof of hypoellipticity can be made using Malliavin calculus it makes sense to ask whether the minorisation condition can be shown directly and quantitatively using Malliavin calculus. This would then allow one to prove hypocoercivity for the SDE quantitatively on the level of the SDE itself rather than via the PDE. Convergence to equilibrium in Wasserstein for the kinetic Fokker-Planck equation is shown very nicely in \cite{EGZ17} by a direct coupling approach. In \cite{EGZ17} they use a Lyapunov structure to show that the solution concentrates in the centre of the state space. Within this centre they show contraction in Wasserstein by using a mixture of reflection and synchronisation couplings. In this setting the reflection coupling should push the $x$ coordinates of the processes towards each other and the synchronisation coupling should push the $v$-coordinates towards each other The final result of this paper is very similar to the one given here. However, our techniques for looking at the behaviour in the centre of the space are very different. We use a much less trajectorial viewpoint. This means we are unlikely to get as sharper constants as with a coupling approach. It does allow us to see how we are exploiting the hypoelliptic structure of the equations more clearly. 

We could not show something as strong as the minorisation condition quantitatively. This is because we use Malliavin calculus to approximate our solutions by Gaussians for which spreading out in all directions is clear but we then get an error from this process which is not bounded in $L^\infty$ as we would need to show minorisation. However this error is sufficiently well behaved that we can bound below the probability that any two solutions to the SDE started within a compact will be within a distance $\delta$ from each other at some time $T$, i.e.
\[ \inf_{ |x|, |y| \leq C} \sup_{\Gamma \in \mathcal{C} ( \mathcal{P}^*_T \delta_x, \mathcal{P}^*_T \delta_y)} \Gamma \{ (x', y') : d(x', y') < \delta \} \geq a. \] Where
\[ \mathcal{C} ( \mathcal{P}^*_T \delta_x, \mathcal{P}^*_T \delta_y) \] is the set of couplings of the solutions at time $T$. This is one of the assumptions of the Wasserstein-1 version of Harris's theorem proved by Hairer and Mattingly in \cite{HM08} to show spectral gaps in Wasserstein for the stochastically forced Navier-Stokes equation.

In order to show the required condition to use Hairer and Mattingly's version of Harris's theorem we need to show that we can construct a coupling so that any two solutions which begin in the centre of the space will move towards each other with positive probability. Since this has to be true for any two solutions our goal is to show that the law of the solutions are spreading out in every direction. It may appear that as noise enters only at the level of velocity in the kinetic Fokker-Planck equation that the law will only spread out in velocity directions. However, the transport operator will mix this to the spatial directions. We need to quantify this effect. Malliavin calculus should help us do this. The Malliavin derivative tells us how the driving Brownian motion affects the solution to the SDE. We use Malliavin calculus to approximate the solution to the SDE by a Gaussian process. This Gaussian process spreads out in all directions, we see the noise passing through iterated commutators of the driving vector fields here. This then allows us to quantitatively verify the hypothesis of Hairer and Mattingly's version of Harris's theorem.

 Therefore the goal is to show exponentially fast convergence to equilibrium in a weighted Wasserstein-1 distance for the kinetic-Fokker Planck or Langevin equation 
\begin{align} \label{SDE}
\mathrm{d}X_t =& V_t \mathrm{d}t, \\ \mathrm{d}V_t =& -(V_t+ \nabla_xU(X_t)) \mathrm{d}t + \sqrt{2} \mathrm{d}W_t.
\end{align} 

The plan of the paper is as follows. We first introduce Hairer and Mattingly's version of Harris's theorem. Then we state our main theorem. We then verify the three assumptions of Hairer and Mattingly's Harris theorem, and show how they contribute to contractivity of the semigroup. The first two assumptions are relatively straight forward though slightly technically involved. For the second assumption we use a version of Bakry-Emery calculus. The third assumption is the key to this proof and where we use Malliavin calculus. We first introduce the tools from the theory of Malliavin calculus for general SDEs before returning to verifying this assumption.
\section{Harris's theorem in Wasserstein}
 We are going to use the version of Harris's theorem in a Wasserstein-1 distance proved by Hairer and Mattingly in \cite{HM08} for use in giving explicit rates of convergence to equilibrium for the 2D Navier-Stokes equation.  We first introduce the distance for some function $L$
\[  \rho_r (x,y) = \inf_\gamma \int_0^1 L^r(\gamma(t)) \|\dot{\gamma}(t) \| \mathrm{d}t,\] where $r$ is an exponent and the infimum runs over all paths $\gamma$ between $x$ and $y$. Let us write $\rho_1 = \rho$.

 The assumptions of this theorem are
\begin{assumption}
There exists a continuous function $L \geq 1$ which has the following properties:

1. There exist strictly increasing functions $L_*, L^*$ such that
\[ L_* (|z|) \leq L(z) \leq L^*(|z|), \] with $\lim_{a \rightarrow \infty} L_*(a) = \infty$.

2. There exist constants $C$ and $\kappa \geq 1$ such that for all $a$
\[ aL^*(a) \leq C L_*^{\kappa}(a). \] 

3. Finally, there exist constants $C_*>0, 0<r_0 < 1$ and a function $\xi : [0,1] \rightarrow [0,1]$ which is non increasing with $\xi (1) < 1$ such that for every $h$ with $|h| =1$ we have
\[ L^r(\Phi_t(z)) ( 1+ \| \nabla_z\Phi_t(z) h \| ) \leq C_* L^{r \xi(t)}(z), \] for every $z$ and every $r \in [r_0, 2 \kappa]$ and every $t \in [0,1]$. Here $\Phi_t$ is the flow map which takes an initial position $z$ to the random variable which is the solution to the SDE at time $t$.
\end{assumption}

\begin{assumption}
There exists a $C_1 >0$ and $p \in [0,1)$ so that for every $\alpha \in (0,1)$ there exists positive $T(\alpha), C(\alpha)$ with 
\[ \|\nabla_z \mathcal{P}_t \phi (z) \| \leq L(z)^p \left( C(\alpha) \sqrt{ (\mathcal{P}_t |\phi|^2)(z)} + \alpha \sqrt{ (\mathcal{P}_t \|\nabla_z\phi\|^2)(z)} \right), \] for every $z \in \mathbb{R}^d$, $\phi \in C^1_b$ and every $t > T(\alpha)$.
\end{assumption}

\begin{assumption}
For any $C>0, r \in (0,1)$ and $\delta > 0$, there exists a $T_0$ so that for any $T \geq T_0$ there exists and $a > 0$ so that 
\[ \inf_{ |z_1|, |z_2| \leq C} \sup_{\pi \in \Pi ( \mathcal{P}^*_T \delta_{z_1}, \mathcal{P}^*_T \delta_{z_2})} \pi \{ (z_1', z_2') : \rho_r(z_1', z_2') < \delta \} \geq a. \] Here $\Pi(\mu, \nu)$ is the set of couplings of $\mu$ and $\nu$. In our situation we actually only use the coupling where they are independent. This depends on $L$ through the distance $\rho$ but not very strongly. We can rewrite this as
\[ \inf_{|z_1|, |z_2| \leq C} \int_{\mathbb{R}^{2d}} \int_{\mathbb{R}^{2d}} 1_{\rho_r (z_1',z_2') < \delta} \mathcal{P}^*_T(\mathrm{d}z_1') \mathcal{P}^*_T(\mathrm{d}z_2') \geq a. \]
\end{assumption}

Then the theorem is
\begin{thm}[Hairer \& Mattingly 2008]\label{wharris}
If the semigroup $\mathcal{P}_t$ satisfies the assumptions above then for all $\mu, \nu$ there exists $C$ and $\lambda$ which we can calculate from the constants in the assumptions so that
\[ \mathcal{W}_\rho ( \mathcal{P}^*_t \mu, \mathcal{P}^*_t \nu ) \leq C e^{-\lambda t} \mathcal{W}_\rho (\mu, \nu), \] for any $\mu, \nu$. Here $\mathcal{W}_\rho$ is the Wasserstein-1 distance corresponding to the distance $\rho$. i.e.
\[ \mathcal{W}_\rho(\mu, \nu) = \inf_{\pi \in \Pi} \int \rho(z_1,z_2) \pi(\mathrm{d}z_1, \mathrm{d}z_2). \] $\Pi$ is the set of couplings of $\mu$ and $\nu$ probability measures on $\mathbb{R}^{2d}$ which have marginals $\mu$ and $\nu$ on the first and last $d$ dimensions.
\end{thm}

Our goal is to verify each of these assumptions with explicit constants. I will briefly describe the strategy. 
\begin{itemize}
\item The first assumption is a Lyapunov structure. We verify this using more tools from \cite{HM08} and known Lyapunov functions for the kinetic Fokker-Planck equation from \cite{HMS02}. 
\item The second assumption is a gradient bound. This is an additional condition needed for the Wasserstein proof to work and is not present in Harris's theorem in any form. We verify this using tools similar to those of Bakry-Emery calculus. Some work on Hypoelliptic diffusions via Bakry-Emery stuff has been done in \cite{B13, M15} and papers referenced therein. We need the Hessian of the confining potential to be bounded for this to work but it seems plausible to relax this assumption. 
\item The third assumption is a kind of uniform boundedness condition. We verify this using Malliavin calculus by showing that for any positive the solution spreads out in all directions. This part should work for any equation satisfying the H\"{o}rmander bracket condition provided that it also satisfies the very strong assumptions that all the vector fields appearing in the commutator conditions are constant.
\end{itemize}
\begin{thm} \label{hypo}
Suppose that $\mathcal{P}_t$ is a semigroup corresponding to the solution to the kinetic Fokker-Plank with the confining potential $U$ being a smooth function satisfying
\[ Hess(U)(x) \leq M, \hspace{10pt} x\cdot \nabla_xU(x) \geq c_1 U(x) +c_2 x^2 - c_3 \] for some strictly positive constants $M, c_1, c_2, c_3$. Then we can choose constants $a_*$ and $k$ depending on these other constants to define the function 
\[ L(x,v) = \exp \left( a_* \left( v^2 +2 U(x) +2 k x^2 + k xv\right)\right). \] We define $\rho$ corresponding to $L$ with
\[ \rho(z_1,z_2) = \inf_{\gamma \in \Gamma} \int_0^1 L(\gamma(t)) \|\dot{\gamma}(t)\| \mathrm{d}t. \] Here $\Gamma$ is the set of all $C^1$ paths between $z_1$ and $z_2$. Then if $\mathcal{W}_\rho$ is the Wasserstein-1 distance associated to $\rho$ we have constants $C>0$ and $\lambda>0$ which we can calculate explicitly such that
\[ \mathcal{W}_\rho(\mathcal{P}_t \mu, \mathcal{P}_t \nu) \leq C e^{- \lambda t} \mathcal{W}_\rho(\mu, \nu). \]
\end{thm}
\begin{remark}
The conditions on $U$ are equivalent to requiring it to behave roughly like a quadratic at infinity. This allows it to have `bad' behaviour on a compact set. For example multiple wells or being flat in large areas. In particular this would allow for the double well potential which behaves quadratically at infinity in 1D.
\end{remark}
\begin{remark}
$\mathcal{W}_\rho(\mu, \nu)$ bounds the Wasserstein 1, distance associated to the euclidean metric. We can see that there exists some $M$ such that
\[ |z_1-z_2| \leq \rho(z_1, z_2) \leq |z_1-z_2| \exp \left( M \left( |z_1|^2+|z_2|^2\right) \right).\]
\end{remark}
We structure the paper as follows. We split the proof of Theorem \ref{hypo} into three parts relating to the three assumptions. We then deal with each of these parts separately. We rely on the theorem of Hairer and Mattingly but in order to make it clear how the proofs work we include a proposition showing how each assumptions will allow us to show contraction for a different part of the space. These propositions follow closely Hairer and Mattingly's proof of theorem \ref{wharris} and are not original. They are intended for expository purposes and to make this chapter more self contained.

\begin{proof}[Proof of \ref{hypo}]
We prove Theorem \ref{hypo} by showing that we can verify all the assumptions of Theorem \ref{wharris} and then applying this result. Assumption 1 is verified in Lemma \ref{lyapunov}. Assumption 2 is verified in Lemma \ref{gradientbound}. Assumption 3 is verified in Lemma \ref{minorisation2}.

We also give the proof of \ref{wharris} in our context. We note that for any distance $d$ we have
\[ \mathcal{W}_{1,d}(\mathcal{P}_t \mu, \mathcal{P}_t \nu) \leq \inf_{\pi \in \Pi(\mu, \nu)} \int \mathcal{W}_{1,d}(\mathcal{P}_t \delta_{z_1}, \mathcal{P}_t \delta_{z_2}) \pi(\mathrm{d}z_1, \mathrm{d}z_2). \] Therefore if we can show for each $z_1,z_2$ that
\[ \mathcal{W}_{1,d}(\mathcal{P}_t \delta_{z_1}, \mathcal{P}_t \delta_{z_2}) \leq \alpha d(z_1,z_2) \] then we have
\[ \mathcal{W}_{1,d}(\mathcal{P}_t\mu, \mathcal{P}_t\nu) \leq \alpha \mathcal{W}(\mu, \nu). \] We do not work directly with the distance $\rho$ and instead look at the equivalent distance
\[ d(z_1, z_2) = \left( \frac{\rho_r(z_1, z_2)}{\delta} \wedge 1\right) + \beta \rho(z_1, z_2). \] For any $r<1$ and $\delta, \beta$ to be chosen later.

In Proposition \ref{contractionl}, we show that there exists some $K$ such that for all $r \in [r_0,1)$ and for all $\beta \in (0,1)$ we have that  $\mathcal{P}_t$ gives a contraction between measures $\delta_{z_1}$ and $\delta_{z_2}$ in $\mathcal{W}_{1,d}$ uniformly over the set $\rho(z_1,z_2) >K$ and uniformly over all $t$ sufficiently large.  In Proposition \ref{contractiong}, we then show that there exists an $r \in [r_0, 1)$ and a $\delta >0$  such that $\mathcal{P}_t$ is a contraction in $\mathcal{W}_{1,d}$ uniformly over the set $\rho_r(z_1,z_2) < \delta$, $\beta \in (0,1)$ and $t$ sufficiently large. Finally in Proposition \ref{contractionm}, we show that for this given $r, \delta$ and $K$ we can choose $\beta$ such that, for every $t$ sufficiently large, $\mathcal{P}_t$ gives a contraction in $\mathcal{W}_{1,d}$ uniformly over the set  $\rho(z_1, z_2) \leq K$ and $\rho_r(z_1, z_2)>\delta$.

\end{proof}
\section{Proofs}
\subsection{Assumption 1}
We would like to show that these assumptions hold with explicit constants for the kinetic Fokker-Planck equation. We begin with assumption 1 where our treatment closely mirrors that of Hairer and Mattingly in \cite{HM08}. Here the Lyapunov function we find is essentially the exponential of the Lyapunov function used by Mattingly, Stuart and Higham in \cite{HMS02}. We write $J_{0,t} = \nabla_z \Phi_{0,t}(z)$

Let us define
\[ Q(x,v) = |v|^2+2U(x) +\frac{1}{2}|x|^2+x \cdot v, \hspace{10pt} P_k(x,v) = 2(|x|^2+|v|^2 + kU(x)). \] We will choose $k$ later.
\begin{lemma} \label{lyapunov1} Let $U$ be a smooth function satisfying that for all $x$ $ x \cdot \nabla_xU(x) \geq c_1 U(x) + c_2 x^2 - c_3$, for strictly positive constants $c_1, c_2, c_3$ and $Hess(U) \leq M$ for some $M>0$.
Define $L_{a}(x,v) = \exp(aQ(x,v))$. Then  we show there exists $a_* >0$ such that, for $0 < a \leq a_*$ and uniformly over $t \in [0,1]$, there is a constant $\beta>0$ such that
\[ \mathbb{E}(L_a(\Phi_t(x,v)) \|J_{0,t}\| ) \leq L_{ae^{-\beta t/4}}(x,v). \] 
\end{lemma}
\begin{proof} Note first that we may as well choose $c_1 \leq 1$.
We have that
\[ \mathrm{d}(a Q(Z_s)) = \left(- a |V_s|^2 -a X_s \cdot \nabla_xU(X_s) +2a\right) \mathrm{d}s +a(X_s+2V_s) \mathrm{d}W_s = -a H_s \mathrm{d}s +a(X_s+2V_s) \mathrm{d}W_s.\] Where 
\[ H_s = |V_s|^2 + X_s \cdot \nabla_x U(X_s) -2. \]
Therefore with $k = c_1$ we have that as functions of $z$
\[ H_s(z) \leq \beta P_k(z)+c_3, \] for some $\beta$ which depends on $c_1,c_2$. We also have that $Q(z) \leq P(z)/c_1$. Now we define
\[ Y_s = e^{\gamma(s-t)}aQ(Z_s) + \gamma \int_0^s e^{\gamma(r-t)}ac_1 P(Z_s), \hspace{10pt} M_s = \int_0^s e^{\gamma(r-t)} a(2V_r+X_r) \mathrm{d}W_r. \]
Differentiating this gives us that,
\[ dY_s = e^{\gamma(s-t)}(aH_s + a\gamma(Q(Z_s)+ c_1P(Z_s)) )\mathrm{d}s + \mathrm{d}M_s.\] Hence for $s<t$ we have
\[ Y_s \leq M_s + Y_0 + a \int_0^s e^{\gamma(r-t)} (H_r + \gamma (Q(Z_r)+c_1 P(Z_r))) \mathrm{d}r \] \[\leq M_s +Y_0 + a\int_0^s e^{\gamma(r-t)} ((2\gamma- \beta)P(Z_r)+2+c_3)\mathrm{d}r. \] Therefore we have that
\[ Y_s \leq M_s +Y_0 + C + a \int_0^s e^{\gamma(r-t)}(2\gamma-\beta)P(Z_r) \mathrm{d}r. \] We now note that we have
\[ Y_0 = ae^{-\gamma t}Q(Z_0), \] and that
\[ Y_t \geq a Q(Z_t) + ac_1\gamma e^{-\gamma t} \int_0^t P(Z_s)\mathrm{d}z. \] We also have that
\[ \langle M\rangle_s \leq 16 a^2 \int_0^s e^{\gamma(r-t)}P(Z_r) \mathrm{d}r,  \] therefore for every $s< t$ we have
\[ M_s - (\beta-2\gamma) c_1 a\int_0^s e^{\gamma(r-t)}P(Z_r) \leq M_s - \frac{c_1(\beta - 2 \gamma)}{16a} \langle M \rangle_s. \] The exponential martingale inequality gives that
\[ \mathbb{P} \left( \sup_{s\leq t} \left( M_s - \frac{c_1( \beta - 2 \gamma)}{16a} \langle M \rangle_s \right) > K \right) \leq \exp \left( - \frac{Kc_1(\beta - 2\gamma)}{8a} \right). \] Now we choose $\gamma = \beta/4$ this gives
\[ Y_s - Y_0 - C \leq M_s - \frac{\beta}{2}ac_1 \int_0^s e^{\gamma(r-t)} P(Z_r) \mathrm{d}r \leq M_s - \frac{c_1\beta}{32 a}\langle M \rangle_s. \] Combining this with our earlier assumptions we have 
\[ aQ(Z_t) + ac_1 \frac{\beta}{4} e^{-\beta t/4} \int_0^t P(Z_s) \mathrm{d}s - ae^{-\beta t/4}Q(Z_0) - C \leq M_s - \frac{\beta c_1}{32 a}. \] Therefore,
\[ \mathbb{P}\left( \exp \left( aQ(Z_t) + ac_1 \frac{\beta}{4} e^{-\beta t/4} \int_0^t P(Z_s) \mathrm{d}s -ae^{-\beta t/4}Q(Z_0) - C\right) > x \right) \leq x^{-c_1\beta/16 a}.  \] We can make $a$ smaller than $a^* = \beta c_1/32$ we have the exponent is bigger than $2$ so we integrate to get
\[ \mathbb{E} \left(\exp \left( aQ(Z_t) + a c_1 \frac{\beta}{4} e^{-\beta t/4} \int_0^t P(Z_s) \mathrm{d}s  - ae^{-\beta t/3} Q(Z_0) -C \right)\right) \leq \frac{c_1 \beta}{c_1 \beta - 16 a}. \] Therefore,
\[\mathbb{E}\left( \exp \left( a Q(Z_t) + ac_1 \frac{\beta}{4} e^{-\beta t/4} \int_0^t P(Z_s) \mathrm{d}s \right)\right) \leq C(a) \exp \left( a Q(Z_0)\right).\]
Now we have that
\[ \mathrm{d}J_{0,t} = \left(\begin{array}{c c} 0 & I \\ - \mbox{Hess}(U)(X_t) & -1\\ \end{array}\right) J_{0,t}\mathrm{d}t. \] It therefore follows that
\[ \mathrm{d}\|J_{0,t} h\| = \frac{(J_{0,t} h)^T}{\|J_{0,t}h\|} \left(\begin{array}{c c} 0 & I \\ - \mbox{Hess}(U)(X_t) & -1\\ \end{array}\right) J_{0,t} h \mathrm{d}t \leq (1+ M) \|J_{0,t} h\|\mathrm{d}t.\] This means that for every $t \in [0,1]$ we have
\[ \| J_{0,t}h\| \leq e^{1+M}. \]Then we have that for $t \in [0,1]$, $h$ a unit vector, $\eta >0$
\[ \|J_{0,t}h\| \leq e^{1+M} \exp \left( \left( \eta\int_0^t (|X_s|^2 +|V_s|^2 + k U(X_s))\mathrm{d}s \right)\right).  \] Therefore for any $a < a_*$ and $\eta$ small enough in terms of $a$ we have
\[ \|J_{0,t} h\| \leq e^{1+M} \exp \left( ac_1 \frac{\beta}{4} e^{-\beta t /4} \int_0^t P(Z_s) \mathrm{d}s \right). \] This combined with our earlier result gives the lemma.
\end{proof}
\begin{lemma} \label{lyapunov} Provided that $U$ is a smooth function satisfying
\[ x \cdot \nabla_xU(x) \geq c_1 U(x) + c_2 x^2 - c_3, \hspace{10pt} HessU(x) \leq M \] for some positive constants we can choose $a_*, k$ such that
\[ L(x,v) = \exp \left( a_* \left( v^2 +2 U(x) +2 k x^2 + k xv\right)\right) \] is a function satisfying assumption 1. 
\end{lemma}
\begin{proof}
We can add a constant in the definition of $U$ so we may as well take $U \geq 0$. Since $Hess(U) \leq M$ we have
\[ \frac{3}{4}(|x|^2 + |v|^2) \leq Q(x,v) \leq (2+M)(|x|^2+|v|^2). \] We also have that 
\[ |z| e^{a(2+M)|z|^2} \leq \frac{1}{a} e^{(3+M)|z|^2} \leq \frac{1}{a} \left( e^{3a|z|^2/4} \right)^{4(3+M)/3} \leq \frac{1}{a}e^{4a(3+M)Q(z)/3}. \]
Therefore if $8a(3+M)/3 \leq a^*$ Then by lemma \ref{lyapunov1} we have that 
\[ \mathbb{E}\left( \left( |\Phi_t(z)| e^{a(2+M)|\Phi_t(z)|^2} \right)^2 \right) \leq \frac{1}{a} e^{4a(3+M) e^{-\beta t/4} Q(z)/3} \] Therefore if we set 
\[ a_* = 3a^*/8(3+M) \] then we can set
\[ L(z)= e^{a_* Q(z)}, L_*(z) = e^{3a_*|z|^2/4}, L^*(z) = e^{(2+M)a_*|z|^2}. \] Then our calculation shows that 
\[  L_* \leq L \leq L^*,\] and furthermore that
\[ |z| L^*(|z|) \leq L_*(z)^\kappa, \] with $\kappa = 3(3+M)/3$. Then lemma \ref{lyapunov1} shows that 
\[ \mathbb{E}( L^r(\Phi_t(z))) \leq L^{re^{-\beta t/4}}(z), \] for all $r \leq 2 \kappa$.
\end{proof}
Now we briefly describe how the proof of Hairer and Mattingly uses this lemma to show convergence for $\rho(z_1, z_2) \geq 4C_1$ with $C_1$ given below.
\begin{prop} \label{contractionl}
If we define $\rho$ as above then for every $\alpha\geq 1/2, T_1>0$ there exists constants $C_1,C$ such that for all $t \geq T_1$
\[ \mathbb{E}(\rho(\Phi_t(z_1), \Phi_t(z_2))) \leq C \rho(z_1,z_2), \]
\[ \mathbb{E}(\rho(\Phi_t(z_1), \Phi_t(z_2))) \leq C_1 + \alpha \rho(z_1, z_2). \]
Furthermore, there exists some radius $R_2$ such that if $|z_1|$ or $|z_2| \geq R_2$ then,
\[ \mathbb{E}(\rho(\Phi_t(z_1), \Phi_t(z_2))) \leq \alpha\rho(z_1, z_2). \]
\end{prop}
\begin{proof}
Fix $z_1, z_2$, $t>T_1$ then there exists some curve joining $z_1, z_2$ such that 
\[ \int_0^1 L^r(\gamma(s))|\dot{\gamma}(s)| \mathrm{d}s \leq \rho_r(z_1, z_2) + \epsilon. \] So then we can evolve every point along this curve by $\Phi_t$ to make a curve joining $\Phi_t(z_1), \Phi_t(z_2)$. Using lemma \ref{lyapunov} this gives
\[ \mathbb{E}(\rho(\Phi_t(z_1), \Phi_t(z_2)) \leq \mathbb{E} \left(\int_0^1 L(\Phi_t(\gamma(s))) |J_{0,t} \dot{\gamma}(s) | \mathrm{d}s  \right) \leq C \int_0^1L(\gamma(s)) |\dot{\gamma}(s)| \mathrm{d}s \leq C(\rho(z_1,z_2) + \epsilon).\] $\epsilon$ was arbitrary. In fact we could have written
\[ \mathbb{E}(\rho(\Phi_t(z_1), \Phi_t(z_2))) \leq C \int_0^1 L^{ e^{-\beta t/4}} (\gamma(s)) |\dot{\gamma}(s)| \mathrm{d}s. \] Then since $L$ grows at infinity there is some $R$ so that $CL^{ e^{-\beta t/4}}(z) \leq \alpha L(z)$ for $|z| \geq R$. Therefore
\[ \mathbb{E} (\rho(\Phi_t(z_1), \Phi_t(z_2))) \leq \alpha \rho(z_1, z_2) + \int_0^1 L(\gamma(s)) |\dot{\gamma(s)}| 1_{\gamma(s) \in B(0,R)} \mathrm{d}s\]
Now we recall that there exists constants $m$ and $M$ so that
\[ Ce^{m|z|^2} \leq L(z) \leq e^{M|z|^2}.\] If we replace the segment of $\gamma$ in $B(0,R)$ by a straight line segment this means we can never need to pick up more than 
\[ Re^{MR^2}+\epsilon \] in our integral while travelling through $B(0,R)$ so we have that 
\[   \mathbb{E}(\rho(\Phi_t(z_1), \Phi_t(z_2))) \leq \alpha \rho(z_1, z_2) + CRe^{MR^2}.\]
So we know we are contractive if $\rho(z_1,z_2) \geq 4C_1$ say, and also we can see from this proof that we will be contractive whenever almost optimal paths between $z_1, z_2$ do not pass through the $B(0,R)$. We can calculate that the distance, $\rho$, from $z$ to $B(0,R)$ is bounded below by 
\[ C \int_R^{|z|} e^{mr^2} \mathrm{d}r. \] Therefore we have $R_2$ such that if $|z|>R_2$ then this will be greater than $4C_1$. This means that if $\gamma$ is a path from $z_1, z_2$ going through $B(0,R)$ with $|z_1|$ or $|z_2|$ greater than $R_2$ then \[\int_\gamma L(\gamma(s)) |\dot{\gamma}(s)| \mathrm{d}s \geq 4C_1. \] This means that if $|z_1| \geq R_2$ or $|z_2| \geq R_2$ then either close to optimal paths do not go through $B(0,R)$ or $\rho(z_1,z_2) \geq 4C_1$. Therefore 
\[ \mathbb{E}(\rho(\Phi_t(z_1), \Phi_t(z_2))) \leq \alpha \rho(z_1, z_2), \] for $|z_1|\geq R_2$ or $|z_2| \geq R_2$.
\end{proof}
\subsection{Assumption 2}
Assumption 2 looks very similar to the gradient bounds found in Malliavin's proof of H\"{o}rmander's theorem see for example \cite{N86}. It seems to be more of a technical challenge than anything else to make the estimates here explicit. However, it is simpler to use more standard hypocoercive techniques based on point wise Bakry-Emery style estimates on the semigroup $\mathcal{P}_t$. Let us write
\[ \Gamma(f,g) = 2\nabla_x f \cdot \nabla_x g - \nabla_x f \cdot \nabla_v g - \nabla_v f \cdot \nabla_x f + 2 \nabla_v f \cdot \nabla_v g.  \]
Now write 
\[ L = \Delta +v \cdot \nabla_x - v \cdot \nabla_v - \nabla_x U \cdot \nabla_v \] this is the forwards operator for the solution to the SDE. We set $\Gamma_2(f) = L \Gamma(f,f) - 2 \Gamma(f, Lf)$. 
\begin{lemma} \label{gradientbound} For $\mathcal{P}_t$ the semigroup associated to the SDE when $U''$ is bounded we have that for an explicit constant $C_M$
\[ |\nabla_x \mathcal{P}_t f|^2 + |\nabla_v \mathcal{P}_t f|^2 \leq  C_M \mathcal{P}_t (f^2) + 3e^{-t/3} \mathcal{P}_t\left( |\nabla_x f|^2 + |\nabla_v f|^2\right). \]
\end{lemma}
\begin{proof}
\begin{align*}
\Gamma_2(f) = & 4| \nabla_x \nabla_v f|^2 -4 \nabla_x \nabla_v f : \nabla_v \nabla_v f +4 |\nabla_v \nabla_v f|^2 +4 \nabla_x f \mbox{Hess}(U) \nabla_v f -2 \nabla_v f \mbox{Hess}(U) \nabla_v f  \\ &+ 2 |\nabla_x f|^2 - 2 \nabla_x f \cdot \nabla_v f +4 |\nabla_v f|^2 -4 \nabla_x f \cdot \nabla_v f\\
\geq & 4 \nabla_x \mbox{Hess}(U) \nabla_v f - 2 \nabla_v \mbox{Hess}(U) \nabla_v f +2 |\nabla_x f|^2 -6 \nabla_x f \cdot \nabla_v f + 4 |\nabla_v f|^2\\
\geq & (2-3 \epsilon_1 -2M \epsilon_2)|\nabla_x f|^2 +\left(4 - \frac{3}{\epsilon_1} - \frac{2M}{\epsilon_2} -2M\right)|\nabla_v f|^2
\end{align*}
We set $ \epsilon_1 = 1/6$ and $ \epsilon_2 = 1/4M$ to get
\[ \Gamma_2(f) \geq |\nabla_x f|^2 - (14 + 6M^2 +2M) |\nabla_v f|^2. \]
Let $\tilde{\Gamma}(f) = \Gamma(f) + (15 +6M^2 +2M) f^2$, and write $C_M = 15 + 6M^2 +2M$. Then we get
\[ L\tilde{\Gamma}(f) - 2 \tilde{\Gamma}(f, Lf) \geq |\nabla_x f|^2 + |\nabla_v f|^2 \geq \frac{1}{3} \Gamma(f) = \frac{1}{3} \left(\tilde{\Gamma}(f) - C_M f^2 \right). \]

Therefore, let
\[ \psi(s) = \mathcal{P}_s \tilde{\Gamma}(\mathcal{P}_{t-s}(f)). \] Then
\begin{align*}
\dot{\psi}(s) \geq \frac{1}{3} \left( \mathcal{P}_s \tilde{\Gamma}( \mathcal{P}_{t-s} f) - C_M \mathcal{P}_s (\mathcal{P}_{t-s})^2 \right) 
\end{align*} Hence,
\[ \frac{\mathrm{d}}{\mathrm{d}s} (e^{-s/3} \psi(s)) \geq - \frac{C_M}{3}e^{-s/3} \mathcal{P}_s ( \mathcal{P}_{t-s}f)^2 \geq - \frac{C_M}{3}e^{-s/3} \mathcal{P}_t (f^2) .\] So
\[ e^{-s/3}\psi(s) - \psi(0) \geq - C_M\left(1-e^{-s/3} \right) \mathcal{P}_t (f^2)  \] which means that
\[ e^{-t/3} \mathcal{P}_t( \Gamma(f)) - \Gamma(\mathcal{P}_t f) - C_M (\mathcal{P}_t f)^2 \geq - C_M \mathcal{P}_t (f^2). \] Rearranging this gives
\[ \Gamma(\mathcal{P}_t f) + C_M ( \mathcal{P}_t f)^2 \leq C_M \mathcal{P}_t(f^2) + e^{-t/3} \mathcal{P}_t ( \Gamma(f)). \] We also have that
\[ |\nabla_x f|^2 + |\nabla_vf|^2  \leq \Gamma(f) \leq 3 \left(|\nabla_x f|^2 + |\nabla_v f|^2 \right).\] So we have that
\[ |\nabla_x \mathcal{P}_t f|^2 + |\nabla_v \mathcal{P}_t f|^2 \leq  C_M \mathcal{P}_t (f^2) + 3e^{-t/3} \mathcal{P}_t\left( |\nabla_x f|^2 + \nabla_v f|^2\right). \]
\end{proof}
Now we look at how this is used to show convergence in the main theorem. We define a new metric
\[ d(z_1, z_2) = \left( \frac{\rho_r(z_1,z_2)}{\delta} \wedge 1 \right) + \beta \rho(z_1, z_2). \]
We see that for $\rho(z_1, z_2) > 4C_1$ proposition \ref{contractionl} still gives a contraction in this metric for every $\beta$.
\begin{prop} \label{contractiong}
If $\rho_r(z_1, z_2) < \delta$ then we have that for $t$ sufficiently large
\[ \mathcal{W}_{1,d}(\mathcal{P}_t \delta_{z_1}, \mathcal{P}_t \delta_{z_2}) \leq \gamma d(z_1,z_2) \] for some explicit $\gamma <1$. 
\end{prop}
\begin{proof}
In this section we want to use the dual Lipschitz formulation of the Wasserstein 1 distance. We have that
\[ \mathcal{W}_{1,d}( \mathcal{P}_t \delta_{z_1}, \mathcal{P}_t \delta_{z_2}) = \sup_\phi (\mathcal{P}_t \phi(z_1) - \mathcal{P}_t \phi(z_2)).\] Here the infimum is taken over all Lipschitz $\phi$ with $|\phi|_{Lip}\leq1$. In fact by density and adding and subtracting we can take a supreme over $phi \in C^1$ with $\phi(0) = 0$. If $\phi$ is such a function then 
\[ |\phi(z)| \leq (1+ \beta)|z|L^*(z), |\nabla\phi(z)| \leq (1/\delta+\beta)L^*(z).  \] Therefore by lemma \ref{gradientbound} and lemma \ref{lyapunov1} we have that
\[ |\nabla \mathcal{P}_t \phi(z)| \leq L^{\kappa e^{-\beta t/4}}(z)(C + 3e^{-t/3}(1/\delta + \beta)) \] Therefore for $t$ sufficiently large so that $\kappa e^{-\beta/4} \leq r$ and $3e^{-t/3} \leq 1/4$ we have that
\[ |\nabla \mathcal{P}_t \phi(z) | \leq (\delta(C+2)+ 1/4) \frac{1}{\delta} L^r(z). \] Now take $\delta \leq 1/2(C+2)$ so we have
\[|\nabla \mathcal{P}_t \phi(z) |  \leq \frac{3}{4} \frac{1}{\delta} L^r(z). \]So we have that
\[ \mathcal{P}_t \phi (z_1) - \mathcal{P}_t \phi(z_2) \leq \int_0^1  \nabla \mathcal{P}_t \phi(\gamma(s)) \cdot \dot{\gamma}(s) \mathrm{d} s \leq \frac{3}{4} \frac{1}{\delta}\int_0^1 L^r(\gamma(s))|\dot{\gamma}(s)|\mathrm{d}s. \] For any path $\gamma$ joining $z_1$ and $z_2$. Therefore we have 
\[ \mathcal{P}_t \phi (z_1) - \mathcal{P}_t \phi(z_2) \leq \frac{3}{4} \frac{1}{\delta} \rho_r(z_1,z_2). \] Since $\rho_r(z_1, z_2) \leq \delta$ this means
\[ \mathcal{W}_{1,d}(\mathcal{P}_t \delta_{z_1}, \mathcal{P}_t \delta_{z_2}) \leq \frac{3}{4} d(z_1,z_2). \]
\end{proof}

\subsection{Assumption 3} Before starting we need some material from Malliavin calculus
\subsubsection{Malliavin Calculus} The material in this section is all standard and follows \cite{N95,H11}. Malliavin calculus is a way of `differentiating' a random variable whose randomness comes from some Brownian motion with respect to this Brownian motion. Since it is the driving Brownian motion which causes the diffusive behaviour of the solutions to SDEs, the Malliavin derivative allows us to measure the strength and direction of this diffusion. We will denote the Malliavin derivative of a function by
$\mathcal{D}F,$ this derivative is in fact a function and if $F$ is a functional of $W_s, 0 \leq s \leq t$ then the Malliavin derivative is a function on $[0,t]$ we denote the evaluation of this function at a particular time $s$ by $ \mathcal{D}_s F.$ We quickly introduce some of the definitions in Malliavin calculus. First we need to know what kind of functions can be differentiated.
Let \[ \Omega = C_0 = \{ f \hspace{5 pt}| \hspace{5pt} f \in C([0,T]^n; \mathbb{R}^d), f(0)=0  \}, \] be Wiener space, and $P$ the Wiener measure. Let $H$ be the Hilbert space $H=L^2([0,T])$. Then we define a simple type of Weiner functional
\[ W: H \rightarrow \mathbb{R}, \hspace{10 pt} W(h) = \int_0^T h(t) \mathrm{d}W_t \] by Ito integration. We have that $\mathcal{D}W(h) = h$. For each $h \in H, W(h)$ is a random variable. Let $\mathscr{G}$ be the sigma-algebra generated by $\{W(h) : h \in H\}$. We want to look a Weiner functionals which are in the Hilbert space $G$,
\[ G= L^2(\Omega, \mathscr{G}, P). \] The Malliavin derivative operator is $\mathcal{D}: G \rightarrow H$ is a closable, unbounded operator much like the weak derivative operator on $L^2$. Since, we are dealing mainly with SDEs we wish to know how to find the Malliavin derivative of the solution to an SDE. If we work purely formally we can derive an SDE for the Malliavin derivative to an SDE, writing in integral form we have
\[ Z_t = Z_0 + \sum_{k=1}^n\int_0^t A_k(Z_s)\mathrm{d}W_{k,s} + \int_0^t B(Z_s)\mathrm{d}s  \] then we can formally take derivatives
\[ \mathscr{D}^k_r Z_t = A_k(Z_r) + \sum_{j=1}^n\int_r^t \nabla A_j (Z_s) \cdot \mathscr{D}^k_r (Z_s) \mathrm{d}W_{j,s} + \int_r^t \nabla B(Z_s) \cdot \mathscr{D}^k_r(Z_s) \mathrm{d}s.\] Here the $k$ in the exponent corresponds to the Malliavin derivative with respect to the $k^{th}$ Brownian motion. The Malliavin derivative can be constructed rigorously and in the case that $A_k$ are smooth and uniformly Lipschitz it can be shown that $\mathscr{D}_r^k$ will satisfy this SDE, see \cite{N95, H11}.

We now wish to look at our solution in a different form. If we write the map
\[ \Phi_{s,t}^\omega (Z_s)=Z_t,  \] the solution map. Then we can differentiate with respect to the initial conditions to get
\[ \partial \Phi_{s,t} = J_{s,t}. \] Then we would like to write an SDE for $J_{s,t}$. Let us write
\[ J_{s,t}Z_s = Z_s + \int_s^t \nabla A_k(Z_r) \cdot J_{s,r}Z_s \mathrm{d}W_{k,r} + \int_s^t \nabla B(Z_r) \cdot J_{s,r} Z_s \mathrm{d}r. \] Comparing this with the SDE for $\mathscr{D}_s Z_t$ shows that, formally anyway, 
\[ \mathscr{D}_s Z_t = J_{s,t}A(Z_s).  \] Furthermore we can write an SDE for $J_{s,t}$ on its own in both Ito and Stratanovich form.
\begin{align*} J_{s,t}&= I + \sum_{k=1}^n\int_s^t \nabla A_{k}(Z_r)\cdot J_{s,r} \mathrm{d}W_{k,r}+\int_s^t \nabla B(Z_r) \cdot J_{s,r} \mathrm{d}r, \\
&= I + \sum_{k=1}^n \int_s^t \nabla A_k (Z_r) \cdot J_{s,r} \circ \mathrm{d}W_{k,r}+ \int_s^t \nabla A_0 (Z_r) \cdot J_{s,r} \mathrm{d}r. \end{align*}
We also notice that as \[ \Phi_{s,t}=\Phi_{r,t} \circ \Phi_{s,r} \] the chain rule gives us that
\[ J_{s,t} = J_{r,t}J_{s,r}. \] We can also show that $J_{s,t}$ is invertible by writing a suitable SDE for $J_{s,t}$ and showing that the solution will not blow up. This lack of blow up comes from global controls on the size of $\nabla A$ and $\nabla B$ which we would like to impose. This SDE is
\[ J_{s,t}^{-1}= I - \sum_{k=1}^n \int_s^t J_{s,r}^{-1}\nabla A_k (Z_r)\circ  \mathrm{d}W_{k,r} - \int_s^t J_{s,r}^{-1} \nabla B (Z_r) \mathrm{d}r.  \] Putting these two facts together gives that
\[ J_{s,t} = J_{0,t}J_{0,s}^{-1} \Rightarrow \mathscr{D}_s Z_t = J_{0,t} J_{0,s}^{-1} A(Z_s). \] This is useful because $J_{0,s}^{-1}A(Z_s)$ is a measurable function of $Z_r, r \leq s$ so we could write an SDE purely on this quantity. This will be useful later, we do this in Stratanovich form where $V$ is any smooth bounded vector field,
\begin{align*}
\circ \mathrm{d}\left( J_{0,t}^{-1} V(Z_t) \right) =& \left(\circ \mathrm{d} J_{0,t}^{-1}  \right) V(Z_t) + J_{0,t}^{-1} \left( \mathrm{d}V(Z_t) \right) \\
=& - \sum_{k=1}^n \nabla A_k(Z_t) J_{0,t}^{-1} V(Z_t) \circ \mathrm{d}W_t^{(k)} - \nabla A_0(Z_t) J_{0,t}^{-1} V(Z_t) \mathrm{d}t \\
&+ J_{0,t}^{-1} \nabla V(Z_t) \left[ \sum_{k=1}^n A_k(Z_t) \circ \mathrm{d}W^{(k)}_t + A_0(Z_t) \mathrm{d}t \right] \\
=& \sum_{k=1}^n J_{0,t}^{-1} [A_k, V](Z_t) \circ \mathrm{d}W^{(k)}_t + J_{0,t}^{-1}[A_0, V] (Z_t)\mathrm{d}t. 
\end{align*}
Converting this to Ito form gives
\[ \mathrm{d}\left(J_{0,t}^{-1} V(Z_t)\right)= \sum_{k=1}^n J_{0,t}^{-1}[A_k,V](Z_t)\mathrm{d}W^{(k)}_t + J_{0,t}^{-1}\left(\frac{1}{2} \sum_{k=1}^n [A_k, [A_k, V]](Z_t) + [A_0,V](Z_t)  \right)\mathrm{d}t. \]

We also need another important theorem from Malliavin calculus
\begin{thm}[Clark-Ocone Representation Formula]
If $F$ is Malliavin differentiable and $\mathbb{E}(F^2) < \infty, \mathbb{E}((\mathcal{D}_s F)^2) < \infty$ and $W$ is a Brownian motion with natural filtration $\mathscr{F}_t$ then,
\[ F= \mathbb{E}(F) + \int_0^t \mathbb{E}(\mathcal{D}_s F | \mathscr{F}_s)\mathrm{d}W_s. \]
\end{thm} This could be considered a version of the fundamental theorem of calculus in this context. A proof of this can be found in \cite{N95}.

\subsubsection{Back to Assumption 3}
Now we return to assumption 3. We are now in the setting of looking the the kinetic Fokker-Planck SDE
\[ \mathrm{d}X_t = V_t \mathrm{d}t, \qquad \mathrm{d}V_t = - V_t \mathrm{d}t - \nabla_x U (X_t) \mathrm{d}t + \mathrm{d}W_t. \] For this SDE we have that
$n=1$ and $A_1 = (0,1)$ and $B = (v, -v - \nabla_x U(x))$. We define $C_1$ by
\[ C_1 := [A_1, B](z) = \left( \begin{array}{c} -1 \\ 1 \end{array} \right).  \] The key idea of this sections is that we can use Malliavin calculus to show that for very small $t$ the solution behaves approximately like
\[ \mathbb{E}(Z_t) + A_1 W_t + C_1 \int_0^t s \mathrm{d}W_s. \] Then because $(W_t, \int_0^t s \mathrm{d}W_s)$ is a $2d$ dimensional non-degenerate Gaussian and because $A_1$ and $C_1$ are linearly independent  this shows that the solution spreads out in every direction. In particular if we take two independent realisations $Z^1_t$ and $Z_t^2$ with different starting points the solutions will spread in the direction $\mathbb{E}(Z^1_t) - \mathbb{E}(Z^2_t)$ which allows us to show there is some positive probability of them becoming close.
\begin{lemma} Let $U$ be smooth and satisfy $Hess(U) \leq M$ and fix $\delta$ and $R$. There exists $T=T(\delta,R)$ such that for fixed $0<t<T$ there exists an $\alpha=\alpha(t,\delta,R)$ with the property that for any two independent solutions to the SDE, $Z^1_t, Z^2_t$ with initial points having $z_1,z_2 \in B(0,R)$, then
\[ \mathbb{P}(|Z_t^1-Z_t^2| < \delta) \geq \alpha.  \]
We have that
\[ \alpha(t, \delta, R) = 1 - C \delta^2 \frac{1}{t_2} \exp\left( -\frac{k}{t^3}m^2 \right)+ 8\exp \left( -\frac{\delta^2}{16Ct^5}\right). \] Here $k$ and $m$ are explicit numerical constants. This value of $\alpha(t,\delta, R)$ is only positive for $t$ sufficiently small and $T$ is the value for which $\alpha(T,\delta, R) = 0$.
\end{lemma}
\begin{proof} The key idea of this proof is to use the fact that the solution spreads out in every direction due to hypoelliptic effects. We represent the solution by a deterministic part, a Gaussian part and a small error. We begin by approximating the Malliavin derivative of the solution using the SDEs 
\begin{align*}
\frac{\mathrm{d}}{\mathrm{d}s} J_{s,t}A_1 &= J_{s,t}C_1,\\ \frac{\mathrm{d}}{\mathrm{d}s}J_{s,t}C_1 &= J_{s,t}C_1 - U''(X_s)J_{s,t}A_1
\end{align*} We can then Taylor expand and use the Clarke-Ocone formula to get
\[ Z_t = \mathbb{E}(Z_t) + \int_0^t \left(\left(\begin{array}{c} 0 \\ 1 \end{array}\right) - (t-s)\left( \begin{array}{c} -1 \\ 1 \end{array}\right) + E_{s,t}\right) \mathrm{d}W_s.  \] 
\[ E_{s,t} = -\mathbb{E} \left(\int_s^t (J_{r,t} C_1 - U''(X_r) J_{r,t} A_1)(t-r)\mathrm{d}r| \mathscr{F}_s\right)\]
 At this point we have to assume that $U''$ is bounded in order to get bounds on $E_{s,t}$. Using the first part with the Lyapunov structure we know that $J_{0,t}$ can be bounded in terms of Lyapunov function we have
 \[ \|J_{s,t}\| \leq e^{1+M} \exp \left(\eta \int_s^t (X_r^2 +V_r^2 +U(X_r) )\mathrm{d}r \right) \leq  e^{1+M} \exp \left(\eta \int_0^t (X_r^2 +V_r^2 +U(X_r) )\mathrm{d}r \right) \leq C L(Z_0).\] Taking the supremum over possible starting points in $B(0,R)$ we have $E_{s,t} \leq C(t-s)^2$ for some constant $C$.  Let us write
\[ \mathcal{E}_t = \int_0^t E_{s,t}\mathrm{d}W_s. \] We would like to get bounds on the expectation of $\exp{c|\mathcal{E}_t|}$. Since 
\[ \mathcal{E}_r = \int_0^r E_{s,t} \mathrm{d}W_s \] is a Martingale for $r \leq t$ then by the exponential martingale inequality 
\[  \mathbb{E} \left( \exp\left( \xi \cdot\mathcal{E}_t \right)\right) \leq \exp \left( \int _0^t C |\xi|^2 (t-s)^4  \mathrm{d}s \right) \leq \exp \left( C |\xi|^2 t^5  \right).\] 
Alternatively, we can bound $J_{s,t}$ in a way that doesn't depend on the initial data but does use that $Hess(U) \leq M$. We can use the equation to see that
\[ |J_{s,t}A_1 + J_{s,t}C_1|^2 \leq 4e^{(2+M)t}. \] Then the rest follows exactly as before but we replace $C$ with $Ce^{(2+M)t}$. Since we are looking at the asymptotics for small $t$ this makes no difference.

So we have decomposed $Z_t$ into a deterministic part $\mathbb{E}(Z_t)$ a Gaussian part which we call $G_t$ and an error which has exponential moments.
\begin{align*} \mathbb{P}\left( Z^1_t -Z^2_t \notin B(0,\delta) \right) \leq & \mathbb{P}\left( \mathbb{E}(Z^1_t)- \mathbb{E}(Z^2_t) + G^1_t-G^2_t \notin B(z, \delta/2) \right) \\ &+ \mathbb{P}\left( \mathcal{E}^1_t \notin B(0, \delta/4)\right)+ \mathbb{P}\left( \mathcal{E}^2 \notin B(0, \delta/4) \right).\end{align*} So we have by Markov's inequality
\[ \mathbb{P}\left( \mathcal{E}_t \notin B(0, \delta/2)\right) \leq 4 \exp \left(C \eta^2 t^5  - \eta \delta/2\right). \] Optimising over $\eta$ gives
\[ \mathbb{P} \left( \mathcal{E}_t \notin B(0, \delta/2) \right) \leq  4 \exp \left( - \frac{\delta^2}{16Ct^5} \right). \] We can write down the density for $G^1_t-G^2_t$. We have
\[ \frac{\mathrm{d}}{\mathrm{d}t} \mathbb{E}(|Z^1_t - Z^2_t|^2) \leq (2+M) \mathbb{E}(|Z^1_t-Z^2_t|^2) + 4d. \] This implies that
\[\mathbb{E}(|Z^1_t - Z_t^2|^2) \leq e^{(2+M)t}(\mathbb{E}(|Z^1_0-Z^2_0|^2)+4d) \leq e^{(2+M)t}(R^2+4d).  \]
We can therefore find the smallest that the density of $G^1_t-G^2_t$ can be on a ball of size $\delta/2$ at the point $-\mathbb{E}(Z_t^1 -Z_t^2)$ when $G^1$ and $G^2$ are independent. Using this we make the two processes independent.
The covariance matrix for $G^1_t, G^2_t$ has eigenvalues $(t/3 + o(t^3), t-t^2+t/3 + o(t^3))$. Lets call $\sigma_m(t)$ the smallest eigenvalue and $\sigma_M(t)$ the largest eigenvalue. We also have that $z \leq L(z)/a_*$ so using Lemma \ref{lyapunov} we have that $\mathbb{E}(Z_t^1-Z_t^2) \leq 2 C_* \max_{|z| \leq R}(L(z))/a_*=:m$ So we can bound the probability by
\[ 1- \delta^2 (2\pi\sigma_M(t))^{d/2}\exp \left( - \frac{m^2}{\sigma_m(t)} \right). \] We then have that
 can approximate for $t \leq 1$,
\[ \mathbb{P} \left( \mathbb{E}(Z^1_t-Z_t^2) + G^1_t-G^2_t \notin B(0, \delta/2)\right) \leq 1- C \delta^2 \frac{1}{t^2} \exp \left( - \frac{k}{t^3}m^2   \right).\] 

Here $k$ and $m$ are constants we can calculate explicitly. In total we have that
\[ \mathbb{P}(Z^1_t-Z^2_t \notin B(0, \delta)) \leq 1 - C \delta^2 \frac{1}{t_2} \exp\left( -\frac{k}{t^3}M^2 \right)+ 8\exp \left( -\frac{\delta^2}{16Ct^5}\right) \]So as $t \rightarrow 0$ we can see that for a fixed sufficiently small $t$ we have
\[ \mathbb{P}( Z_t \in B(z, \delta)) \geq \alpha. \] Where we can calculate $\alpha$ explicitly in terms of $t, \delta, R$ and the other constants appearing in the equation.
\end{proof}
\begin{lemma} \label{minorisation2}
Suppose we fix $\delta, t$ and $R$. Then there exists $\alpha$ such that for any two independent solutions to the SDEs $Z^1_t, Z^2_t$ with initial points having $z_1-z_2 \in B(0,R)$ then
\[ \mathbb{P}(|Z_t^1-Z_t^2| < \delta) \geq \alpha.  \] Furthermore if they start with initial points both in $B(0,R)$ then
\[ \mathbb{P}(\rho_r(Z_t^1, Z_t^2) < \delta) \geq \alpha'. \]
\end{lemma}
\begin{proof}
 We want to extend the previous Lemma to larger times by showing that if two solutions start with $z^1-z^2 \in B(0,R)$ then they stay there with some positive probability. To do this we repeat the calculation but replacing $\delta$ by $R$ then since the two processes are independent the probability that their difference stay inside $B(0,R)$ is given by the first lemma. So we have for some $t_*$
\[ \mathbb{P}(Z^1_{t_*}-Z^2_{t_*} \in B(0, R) \hspace{5pt} | \hspace{5pt} Z^1_0- Z^2_0 \in B(0,R)) \geq b \] Therefore if $t = nt_* + s$ with $s \leq t_*$ then 
\[ \mathbb{P}( Z^1_t - Z^2_t \in B(0, \delta)\hspace{5pt}|\hspace{5pt} Z^1_t-Z^2_t \in B(0,R)) \geq ab^n. \] Here $a,b, t_*$  are explicitly calculable constants depending on $M, F$.
However, we in fact need to look at $\rho_r$ instead of the normal distance. In order to do this we need to look at 
\[ \mathbb{P}(Z^1_t-Z^2_t \in B(0, \delta), Z_t^1, Z_t^2 \in B(0, R')), \] for some $R'$. We have that
\begin{align*} \mathbb{P}( Z^1_t - Z^2_t \in B(0, \delta), Z^1_t, Z^2_t \in B(0, R')) =& \mathbb{P}( Z_t^1, Z_t^2 \in B(0,R'))\\ & - \mathbb{P}( Z^1_t - Z^2_t \notin B(0, \delta), Z^1_t, Z_t^1 \in B(0,R')) \end{align*}
So we bound
\begin{align*} \mathbb{P}( Z_t^1 -Z_t^2 \notin B(0, \delta), Z_t^1, Z_t^2 \in B(0,R')) \leq &\mathbb{P}(Z_t^1, Z_t^2 \in B(0,R)) - \\ &\mathbb{P}( \mathbb{E}(Z_t^1 -Z_t^2) + G_t^1 -G^2_t \in B(0, \delta/2), Z^1_t, Z^2_t \in B(0,R'))  \\&+ \mathbb{P}(\|E^1_t\| \leq \delta/4) + \mathbb{P}( \| E^2_t\| \leq \delta/4)\end{align*} Furthermore we have
\[\mathbb{P}( \mathbb{E}(Z_t^1 -Z_t^2) + G_t^1 -G^2_t \in B(0, \delta/2), Z^1_t, Z^2_t \in B(0,R')) \geq C \delta^2 R' \frac{1}{t^2} \exp \left( -K/t^3\right) \] for explicitly computable constants $C$ and $K$. So in the same way we have for all $t, R, R'$ there is $a(t,R,R', \delta) >0$ such that
\[ \mathbb{P}( Z^1_t- Z^2_t \in B(0, \delta), Z^1_t, Z^2_t \in B(0,R') \hspace{5pt}|\hspace{5pt} Z^1_0, Z^2_0 \in B(0,R)) \geq a(t,R,R',\delta). \] Then we can find an $R''$ such that on any optimal path between two points in $B(0,R')$ we have $L(\gamma(t)) \leq R''$ so this implies that for $x,y \in B(0,R')$ we have
\[\rho_r(x,y) = \inf_\gamma \int_0^t L(\gamma(t))\dot{\gamma}(t) \mathrm{d}t \leq \inf_\gamma R'' \int_0^t \dot{\gamma}(t) \mathrm{d}t = R''|x-y|.\] We mean that the two distances are equivalent on compact sets. So if $|x-y| \leq \delta/R''$ we have that $\rho_r(x,y) \leq \delta$ therefore

\[ \mathbb{P}( \rho_r(Z^1_t, Z^2_t) \leq \delta \hspace{5pt}|\hspace{5pt} Z^1_0, Z^2_0 \in B(0, R)) \geq a(t, R, R', \delta/R''). \]
\end{proof}
Now for this section we look again at how this shows contraction in the theorem of Hairer and Mattingly. We have that 
\begin{prop}\label{contractionm}
If $\rho(z_1, z_2) \leq 4C_1$ and $\rho_r(z_1, z_2) > \delta$ then there exists $\gamma$ such that
\[ \mathcal{W}_{1,d}(\mathcal{P}_t \delta_{z_1}, \mathcal{P}_t \delta_{z_2}) \leq \gamma d(z_1,z_2). \] 
\end{prop}
\begin{proof}
Suppose that we have that $\rho(z_1,z_2)_r \geq \delta$ and $\rho(z_1, z_2) \leq 4C_1$ then we have that $z_1,z_2$ are contained in some ball. There is some $R$ such that for $|z|\geq R$ we have 
\[ L^*(z)^r \leq \frac{\delta}{8C_1}L_*(z). \] Then as we discussed there is some $R'$ such that
\[ \int_R^{R'} L_*(r) \mathrm{d}r \geq 8C_1. \] Therefore if $|z_1|, |z_2|\geq R'$ and $\rho(z_1,z_2) \leq 4C_1$ then if $\gamma$ is a path such that
\[ \int_0^1 L(\gamma(s)) |\dot{\gamma}(s)|\mathrm{d}s \leq \rho(z_1, z_2) + \epsilon \] then $\gamma$ must not pass through $B(0,R)$. and for such a path 
\[ \rho_r(z_1,z_2) \leq \int_0^1 L^r(\gamma(s))|\dot{\gamma(s)}|\mathrm{d}s \leq \frac{\delta}{8C_1} \int_0^1 L(\gamma(s))|\dot{\gamma}(s)| \mathrm{d}s \leq \frac{\delta}{8C_1}(4C_1 + \epsilon). \] Since $\epsilon$ is arbitrary this shows that $\rho(z_1, z_2) \leq \delta$. Therefore if $\rho(z_1, z_2) \leq 4C_1$ and $\rho_r(z_1, z_2) \geq \delta$ we have that $z_1, z_2 \in B(0,R')$. Then for this $R'$ we can apply lemma \ref{minorisation2} to get that there is some $a$ such that if we make $Z^1, Z^2$ independent then we have
\[ \mathbb{P}( \rho_r(Z^1_t, Z^2_t) \leq \delta/2 \hspace{5pt}|\hspace{5pt} Z^1_0, Z^2_0 \in B(0, R')) \geq a. \] Using this we have for the independent coupling
\[\mathbb{E}( d(Z^1_t, Z^2_t)) \leq \frac{1}{2}\mathbb{P}(\rho(Z^1_t, Z^2_t) \leq \delta/2) + (1- \mathbb{P}(\rho(Z^1_t, Z^2_t) \leq \delta/2) )+ \beta \mathbb{E}(\rho(Z^1_t, Z^2_t)) \]
\[ \leq (1-a/2) + \beta (\mathbb{E}(\rho(0,Z^1_t)) + \mathbb{E}(\rho(0,Z^1_t))). \]
Now we can see that 
\[ \mathbb{E}(\rho(0, Z^1_t)) \leq \mathbb{E}(|Z^1_t|L^*(Z^1_t)) \leq CL^{\kappa}(z_1) \leq C_* \] since $z_1 \in B(0,R')$. So if we take $\beta \leq a/8C_*$ then we have that
\[  \mathbb{E}( d(Z^1_t, Z^2_t)) \leq 1-a/4 \leq (1-a/4) d(z_1, z_2).\]
\end{proof}
\bibliographystyle{abbrv}
\bibliography{thesis}{}
\end{document}